\documentclass[12pt,twoside]{amsart}
\usepackage{amsmath, amsthm, amscd, amsfonts, amssymb, graphicx}
\usepackage{enumerate}
\usepackage[colorlinks=true,
linkcolor=magenta,
urlcolor=cyan,
citecolor=cyan]{hyperref}

\newtheorem{theorem}{Theorem}
\newtheorem{lemma}{Lemma}
\newtheorem{remark}{Remark}

\newtheorem{corollary}{Corollary}

\newtheorem{proposition}{Proposition}
\numberwithin{equation}{section}

\begin{document}
\title{A note around operator Bellman inequality}
\author{Shiva Sheybani$^1$, Mohsen Erfanian Omidvar$^2$, and Mahnaz Khanehgir$^3$}
\maketitle

\begin{abstract}
In this paper, we shall give an extension of operator Bellman inequality. This result is estimated via Kantorovich constant.
\end{abstract}

\pagestyle{myheadings} 
\markboth{\centerline {A note around operator Bellman inequality}}
{\centerline {S. Sheybani, M.E. Omidvar \& M. Khanehgir}} \bigskip \bigskip 

\section{Introduction}
\vskip 0.4 true cm
Let $\mathcal{B}(\mathcal{H})$ be the $C^*$ algebra of bounded linear operators on a complex Hilbert space $\mathcal{H}.$ If $A\in\mathcal{B}(\mathcal{H})$ is  positive, we write  denote $A\ge 0.$ For two self-ajoint operators $A,B\in\mathcal{B}(\mathcal{H})$, we  write  $A\le B$ if $B-A\geq 0$. For a real-valued function $f$ of a real variable and a self-adjoint operator $A\in \mathcal{B}\left( \mathcal{H} \right)$, the value $f\left( A \right)$ is understood by means of the functional calculus.\\
Let $J$ be a real interval of any type. A continuous function $f:J\to \mathbb{R}$ is said to be operator concave if $f\left( A{{\nabla }_{v}}B \right)\ge f\left( A \right){{\nabla }_{v}}f\left( B \right)$ holds for each $v\in \left[ 0,1 \right]$ and every pair of self-adjoint operators $A,B\in \mathcal{B}\left( \mathcal{H} \right)$, with spectra in $J$.

A linear map $\Phi :\mathcal{B}\left( \mathcal{H} \right)\to \mathcal{B}\left( \mathcal{K} \right)$ is said to be positive if $\Phi \left( A \right)\ge 0$ when $A\geq 0$. If, in addition, $\Phi \left( I \right)=I$, it is said to be normalized.

Bellman \cite{6} showed that if $n,r$ are positive integers and $A,B,{{a}_{i}},{{b}_{i}}\left( 1\le i\le n \right)$ are positive real numbers such that $\sum\nolimits_{i=1}^{n}{a_{i}^{r}}\le {{A}^{r}}$ and $\sum\nolimits_{i=1}^{n}{b_{i}^{r}}\le {{B}^{r}}$, then
\[{{\left( {{A}^{r}}-\sum\limits_{i=1}^{n}{a_{i}^{r}} \right)}^{\frac{1}{r}}}+{{\left( {{B}^{r}}-\sum\limits_{i=1}^{n}{b_{i}^{r}} \right)}^{\frac{1}{r}}}\le {{\left( {{\left( A+B \right)}^{r}}-\sum\limits_{i=1}^{n}{{{\left( {{a}_{i}}+{{b}_{i}} \right)}^{r}}} \right)}^{\frac{1}{r}}}.\]
The operator Bellman inequality \cite[Corollary 2.2]{5} asserts that: If $\Phi $ is a normalized positive linear map on $\mathcal{B}\left( \mathcal{H} \right)$, $A, B$ are contractions (in the sense that $\left\| A \right\|,\left\| B \right\|\le 1$), then
\begin{equation}\label{10}
\Phi \left( {{\left( I-A \right)}^{r}}{{\nabla }_{v}}{{\left( I-B \right)}^{r}} \right)\le \Phi {{\left( I-A{{\nabla }_{v}}B \right)}^{r}}
\end{equation}
where $r,v\in \left[ 0,1 \right]$.

By applying similar method presented in \cite{5}, we infer that
\[\Phi {{\left( I-A{{\nabla }_{v}}B \right)}^{r}}\le \Phi \left( {{\left( I-A \right)}^{r}}{{\nabla }_{v}}{{\left( I-B \right)}^{r}} \right),\text{ }r\in \left[ -1,0 \right]\cup \left[ 1,2 \right].\]
Of course, the above inequality does not hold in general  when  $r\notin \left[ -1,0 \right]\cup \left[ 1,2 \right]$. For instance,  if we take $\Phi \left( X \right)=X$, $A=\left( \begin{matrix}
2 & 1  \\
1 & 1  \\
\end{matrix} \right)$, $B=\left( \begin{matrix}
1 & 0  \\
0 & 0  \\
\end{matrix} \right)$, $v=\frac{1}{2}$, and $f\left( t \right)={{t}^{3}}$. By a simple calculation, we have
\[\left( \begin{matrix}
-0.25 & -0.25  \\
-0.25 & 0.25  \\
\end{matrix} \right)=\Phi {{\left( I-A{{\nabla }_{v}}B \right)}^{r}}\nleq	 \Phi \left( {{\left( I-A \right)}^{r}}{{\nabla }_{v}}{{\left( I-B \right)}^{r}} \right)=\left( \begin{matrix}
-1.5 & -1  \\
-1 & 0  \\
\end{matrix} \right).\] 

Very recently, as an extension of inequality \eqref{10} to the negative parameter, the authors in \cite[Theorem 3.2]{4} proved that for any contraction operators $A,B\in \mathcal{B}\left( \mathcal{H} \right)$, $v\in \left[ 0,1 \right]$, and $r\in \left[ -1,0 \right]$
\[\begin{aligned}
 \Phi {{\left( I-A{{\nabla }_{v}}B \right)}^{r}}&\le \Phi {{\left( I-A \right)}^{r}}{{\sharp}_{v}}\Phi {{\left( I-B \right)}^{r}} \\ 
& \le \Phi \left( {{\left( I-A \right)}^{r}}{{\sharp}_{v}}{{\left( I-B \right)}^{r}} \right) \\ 
& \le \Phi \left( {{\left( I-A \right)}^{r}}{{\nabla }_{v}}{{\left( I-B \right)}^{r}} \right)  
\end{aligned}\]
holds, where the notation ${{\sharp}_{v}}$ used for the weighted geometric mean between two positive operators, and is defined as follows
\[A{{\sharp}_{v}}B={{A}^{\frac{1}{2}}}{{\left( {{A}^{-\frac{1}{2}}}B{{A}^{-\frac{1}{2}}} \right)}^{v}}{{A}^{\frac{1}{2}}}\qquad\text{ }v\in \left[ 0,1 \right].\]

In this paper, we aim to provide a generalization of the inequality \eqref{10} to $r\notin \left[ -1,0 \right]\cup \left[ 1,2 \right]$. Additionally, we show a reverse Bellman type inequality of additive type, by using some ideas from \cite{5,4}.
\section{Main Results}
\vskip 0.4 true cm
In order to prove our theorem, we need the following lemmas.
\begin{lemma}\label{8}
	\cite[Theorem 1.2]{1} If $f:J\to \mathbb{R}$ is a concave function, then
\[\left\langle f\left( A \right)x,x \right\rangle \le f\left( \left\langle Ax,x \right\rangle  \right)\]
for any self-adjoint operator $A$ with spectra contained in $J$  and any unit vector $x\in \mathcal{H}$.
\end{lemma}
\begin{lemma}
\cite[Corollary 4.12]{2}) Let $A\in \mathcal{B}\left( \mathcal{H} \right)$ such that $m\le A\le M$ and $\Phi $ be a normalized positive linear map. If $f:\left[ m,M \right]\to \mathbb{R}$ is a concave function, then
\begin{equation}\label{1}
K\left( m,M,f \right)\Phi \left( f\left( A \right) \right)\le f\left( \Phi \left( A \right) \right)\le \frac{1}{K\left( m,M,f \right)}\Phi \left( f\left( A \right) \right)
\end{equation}
where
\begin{equation}\label{2}
K\left( m,M,f \right)=\max \left\{ \frac{1}{f\left( t \right)}\left( \frac{M-t}{M-m}f\left( m \right)+\frac{t-m}{M-m}f\left( M \right) \right):\text{ }m\le t\le M \right\}.
\end{equation}
\end{lemma}
We can now state our first main result.
\begin{proposition}\label{14}
Let $A,B\in \mathcal{B}\left( \mathcal{H} \right)$ be two self-adjoint operators such that $m\le A,B\le M$ and $\Phi $ be normalized positive linear map. If $f:\left[ m,M \right]\to \mathbb{R}$ is a concave function, then for any $v\in \left[ 0,1 \right]$
\begin{equation}\label{3}
\Phi \left( f\left( A \right) \right){{\nabla }_{v}}\Phi \left( f\left( B \right) \right)\le \frac{1}{K{{\left( m,M,f \right)}^{2}}}f\left( \Phi \left( A{{\nabla }_{v}}B \right) \right)
\end{equation}
where $K\left( m,M,f \right)$ is defined as \eqref{2}.
\end{proposition}
\begin{proof} 
The assumption $m\le A,B\le M$ implies $m\le \left\langle A{{\nabla }_{v}}Bx,x \right\rangle \le M$ for any $x\in \mathcal{H}$ with $\left\| x \right\|=1$. So
\[\begin{aligned}
 \left\langle f\left( A \right){{\nabla }_{v}}f\left( B \right)x,x \right\rangle &\le f\left( \left\langle Ax,x \right\rangle  \right){{\nabla }_{v}}f\left( \left\langle Bx,x \right\rangle  \right) \quad \text{(by Lemma \ref{8})}\\ 
& \le f\left( \left\langle A{{\nabla }_{v}}Bx,x \right\rangle  \right) \quad \text{(since $f$ is concave)}\\ 
& \le \frac{1}{K\left( m,M,f \right)}\left\langle f\left( A{{\nabla }_{v}}B \right)x,x \right\rangle   \quad \text{(by RHS of \eqref{1})}
\end{aligned}\]
i.e.,
\[f\left( A \right){{\nabla }_{v}}f\left( B \right)\le \frac{1}{K\left( m,M,f \right)}f\left( A{{\nabla }_{v}}B \right)\]
 which is an inequality of interest in itself. The hypothesis on $\Phi $ ensures that 
\[\begin{aligned}
 \Phi \left( f\left( A \right){{\nabla }_{v}}f\left( B \right) \right)&=\Phi \left( f\left( A \right) \right){{\nabla }_{v}}\Phi \left( f\left( B \right) \right) \\ 
& \le \frac{1}{K\left( m,M,f \right)}\Phi \left( f\left( A{{\nabla }_{v}}B \right) \right) \\ 
& \le \frac{1}{K{{\left( m,M,f \right)}^{2}}}f\left( \Phi \left( A{{\nabla }_{v}}B \right) \right)  \quad \text{(by LHS of \eqref{1})}
\end{aligned}\]
which is the desired inequality \eqref{3}.
\end{proof}
\begin{proposition}\label{9}
Let all the assumptions of Proposition \ref{14} hold except the condition concavity which is changed to convexity. Then
\[\frac{1}{K{{\left( m,M,f \right)}^{2}}}f\left( \Phi \left( A{{\nabla }_{v}}B \right) \right)\le \Phi \left( f\left( A \right) \right){{\nabla }_{v}}\Phi \left( f\left( B \right) \right).\]
\end{proposition}

Letting $f\left( t \right)={{\left( 1-t \right)}^{r}}\left( t<1, r\notin [-1,0]\cap[1,2] \right)$ in Proposition \ref{9}, we infer the following result:

\begin{theorem}\label{11}
Let $A,B\in \mathcal{B}\left( \mathcal{H} \right)$ be two contractions operators such that $m\le A,B\le M$ and $\Phi $ be normalized positive linear map. Then for any $v\in[0,1]$ and $r\notin  [-1,0]\cap[1,2]$
\begin{equation}\label{12}
{{\left( I-\Phi \left( A{{\nabla }_{v}}B \right) \right)}^{r}}\le K{{\left( m,M,{{\left( 1-t \right)}^{r}} \right)}^{2}}\Phi \left( {{\left( I-A \right)}^{r}}{{\nabla }_{v}}{{\left( I-B \right)}^{r}} \right).
\end{equation}

\end{theorem}
\begin{remark}
	In the following we aim to find the value of $K{{\left( m,M,{{\left( 1-t \right)}^{r}} \right)}^{2}}$ in Theorem \ref{11}. To do this end, it is enough we find the maximum of the function $h\left( t \right)\equiv {{\left( \frac{\mu t+\lambda }{{{\left( 1-t \right)}^{r}}} \right)}^{2}}$ for  $0<m\le t\le M<1$ and $r\notin \left[ 1,2 \right]\cup \left[ -1,0 \right]$ with $\mu =\frac{{{\left( 1-M \right)}^{r}}-{{\left( 1-m \right)}^{r}}}{M-m}$ and $\lambda =\frac{M{{\left( 1-m \right)}^{r}}-m{{\left( 1-M \right)}^{r}}}{M-m}$.  Notice that $\left\{ \begin{array}{rr}
	 \mu <0,\lambda >0&\text{ for }r>2 \\ 
	 \mu >0,\lambda <0&\text{ for }r<-1 \\ 
	\end{array} \right.$. By an easy computation we  infer
\[h'\left( t \right)=-\frac{2\left( \mu t+\lambda  \right)\left( \left( r-1 \right)\mu x+\lambda r+\mu  \right)}{{{\left( 1-t \right)}^{2r}}\left( t-1 \right)},\] 
and
\[h'\left( t \right)=0\text{ }\Rightarrow \text{ }{{t}_{0}}=-\frac{\lambda }{\mu }\quad\&\quad{{t}_{1}}=-\frac{\lambda r+\mu }{\mu \left( r-1 \right)}.\]
It is not hard to check that $h(t)$ attains its maximum at $t_1$ provided that $m\le t_1\le M$. To see the proof of inequalities ${{t}_{1}}-m\ge 0$ and $M-{{t}_{1}}\ge 0$ we refer the reader to \cite{smf}.
\end{remark}
\begin{remark}
Assume $n$ is a positive integers, $r>2$, $v\in \left[ 0,1 \right]$ and ${{a}_{k}},{{b}_{k}}\left( 1\le k\le n \right)$ are positive real numbers such that $0<m\leq\sum\nolimits_{k=1}^{n}{a_{k}^{\frac{1}{r}}}\le M< 1$ and $0<m\leq\sum\nolimits_{k=1}^{n}{b_{k}^{\frac{1}{r}}}\le M< 1$. Take $A=\left( \begin{matrix}
\sum\nolimits_{k=1}^{n}{a_{k}^{\frac{1}{r}}} & 0  \\
0 & 1  \\
\end{matrix} \right)\in {{M}_{2}}\left( \mathbb{C} \right)$, $B=\left( \begin{matrix}
\sum\nolimits_{k=1}^{n}{b_{k}^{\frac{1}{r}}} & 0  \\
0 & 1  \\
\end{matrix} \right)\in {{M}_{2}}\left( \mathbb{C} \right)$.    Consequently, 
\[\begin{aligned}
 {{\left( I-A{{\nabla }_{v}}B \right)}^{r}}&={{\left( \left( \begin{matrix}
		1 & 0  \\
		0 & 1  \\
		\end{matrix} \right)-\left( \begin{matrix}
		\left( 1-v \right)\sum\nolimits_{k=1}^{n}{a_{k}^{\frac{1}{r}}+v\sum\nolimits_{k=1}^{n}{b_{k}^{\frac{1}{r}}}} & 0  \\
		0 & 1  \\
		\end{matrix} \right) \right)}^{r}} \\ 
& =\left( \begin{matrix}
\left( 1-\sum\nolimits_{k=1}^{n}{a_{k}^{\frac{1}{r}}} \right){{\nabla }_{v}}{{\left( 1-\sum\nolimits_{k=1}^{n}{b_{k}^{\frac{1}{r}}} \right)}^{r}} & 0  \\
0 & 0  \\
\end{matrix} \right),  
\end{aligned}\]
and
\[\begin{aligned}
 {{\left( I-A \right)}^{r}}{{\nabla }_{v}}{{\left( I-B \right)}^{r}}&=\left( 1-v \right){{\left( \begin{matrix}
		1-\sum\nolimits_{k=1}^{n}{a_{k}^{\frac{1}{r}}} & 0  \\
		0 & 0  \\
		\end{matrix} \right)}^{r}}+v{{\left( \begin{matrix}
		1-\sum\nolimits_{k=1}^{n}{b_{k}^{\frac{1}{r}}} & 0  \\
		0 & 0  \\
		\end{matrix} \right)}^{r}} \\ 
& =\left( \begin{matrix}
{{\left( 1-\sum\nolimits_{k=1}^{n}{a_{k}^{\frac{1}{r}}} \right)}^{r}}{{\nabla }_{v}}{{\left( 1-\sum\nolimits_{k=1}^{n}{b_{k}^{\frac{1}{r}}} \right)}^{r}} & 0  \\
0 & 0  \\
\end{matrix} \right).  
\end{aligned}\]
Now,
\[\begin{aligned}
 {{\left( 1-\sum\limits_{k=1}^{n}{{{\left( {{a}_{k}}{{\nabla }_{v}}{{b}_{k}} \right)}^{\frac{1}{r}}}} \right)}^{r}}&\le {{\left( 1-\sum\limits_{k=1}^{n}{\left( a_{k}^{\frac{1}{r}}{{\nabla }_{v}}b_{k}^{\frac{1}{r}} \right)} \right)}^{r}} \quad \text{(since $f\left( t \right)={{t}^{\frac{1}{r}}}\left( r>2 \right)$ is concave)}\\ 
& ={{\left( 1-\left( \sum\limits_{k=1}^{n}{a_{k}^{\frac{1}{r}}}{{\nabla }_{v}}\sum\limits_{k=1}^{n}{b_{k}^{\frac{1}{r}}} \right) \right)}^{r}} \\ 
& ={{\left( \left( 1-\sum\limits_{k=1}^{n}{a_{k}^{\frac{1}{r}}} \right){{\nabla }_{v}}\left( 1-\sum\limits_{k=1}^{n}{b_{k}^{\frac{1}{r}}} \right) \right)}^{r}} \\ 
& \le \xi \left( {{\left( 1-\sum\limits_{k=1}^{n}{a_{k}^{\frac{1}{r}}} \right)}^{r}}{{\nabla }_{v}}{{\left( 1-\sum\limits_{k=1}^{n}{b_{k}^{\frac{1}{r}}} \right)}^{r}} \right) \quad \text{(by Corollary \ref{11})}. 
\end{aligned}\]
Take $v=\frac{{{M}_{2}}}{{{M}_{1}}+{{M}_{2}}}$ and substitute ${{a}_{k}}$ and ${{b}_{k}}$ by $\frac{{{a}_{k}}}{{{M}_{1}}}$ and $\frac{{{b}_{k}}}{{{M}_{2}}}$, respectively. We infer that
\[\begin{aligned}
& {{\left( 1-\sum\limits_{k=1}^{n}{{{\left( \left( 1-\frac{{{M}_{2}}}{{{M}_{1}}+{{M}_{2}}} \right)\left( \frac{{{a}_{k}}}{{{M}_{1}}} \right)+\left( \frac{{{M}_{2}}}{{{M}_{1}}+{{M}_{2}}} \right)\left( \frac{{{b}_{k}}}{{{M}_{2}}} \right) \right)}^{\frac{1}{r}}}} \right)}^{r}} \\ 
& \le \xi \left( \left( 1-\frac{{{M}_{2}}}{{{M}_{1}}+{{M}_{2}}} \right){{\left( 1-\sum\limits_{k=1}^{n}{{{\left( \frac{{{a}_{k}}}{{{M}_{1}}} \right)}^{\frac{1}{r}}}} \right)}^{r}}+\frac{{{M}_{2}}}{{{M}_{1}}+{{M}_{2}}}{{\left( 1-\sum\limits_{k=1}^{n}{{{\left( \frac{{{b}_{k}}}{{{M}_{2}}} \right)}^{\frac{1}{r}}}} \right)}^{r}} \right).\\ 
\end{aligned}\]
Equivalently,
\[\begin{aligned}
& \frac{1}{{{M}_{1}}+{{M}_{2}}}{{\left( {{\left( {{M}_{1}}+{{M}_{2}} \right)}^{\frac{1}{r}}}-\sum\limits_{k=1}^{n}{{{\left( {{a}_{k}}+{{b}_{k}} \right)}^{\frac{1}{r}}}} \right)}^{r}} \\ 
& \le \xi \left( \frac{1}{{{M}_{1}}+{{M}_{2}}}{{\left( M_{1}^{\frac{1}{r}}-\sum\limits_{k=1}^{n}{a_{k}^{\frac{1}{r}}} \right)}^{r}}+\frac{1}{{{M}_{1}}+{{M}_{2}}}{{\left( M_{2}^{\frac{1}{r}}-\sum\limits_{k=1}^{n}{b_{k}^{\frac{1}{r}}} \right)}^{r}} \right).
\end{aligned}\]

\end{remark}

By choosing $f\left( t \right)=\exp \left( t \right)$ in Proposition \ref{9} we have:
\begin{corollary}
Let $A,B\in \mathcal{B}\left( \mathcal{H} \right)$ be two self-adjoint operators such that $m\le A,B\le M$ and $\Phi $ be normalized positive linear map. Then
\[\exp \left( \Phi \left( A{{\nabla }_{v}}B \right) \right)\le K{{\left( m,M,\exp \left( t \right) \right)}^{2}}\Phi \left( \exp \left( A \right){{\nabla }_{v}}\exp \left( B \right) \right).\]
\end{corollary}

\begin{remark}
To find the value of $K{{\left( m,M,\exp \left( t \right) \right)}^{2}}$ we set $h\left( t \right)\equiv {{\left( \frac{\mu t+\lambda }{\exp t} \right)}^{2}}$ for  $0<m\le t\le M<1$ with $\mu =\frac{\exp M-\exp m}{M-m}$ and $\lambda =\frac{M\exp m-m\exp M}{M-m}$. It is easy to see that $\lambda ,\mu >0$, and
\[h'\left( t \right)=-\frac{2\left( \mu t+\lambda  \right)\left( \mu t+\lambda -\mu  \right)}{\exp 2x}.\]
Solving $h'(t)=0$, we obtain
\[{{t}_{0}}=-\frac{\lambda }{\mu }\text{  }\!\!\;\And\;\!\!\text{  }{{t}_{1}}=-\frac{\lambda -\mu }{\mu }.\]
After simple computations, we find that $h(t)$ attains its maximum at $t_1$ provided that $m\le t_1\le M$. For details we refer to \cite{1}.
\end{remark}
In the following, we aim to prove a complementary inequality for the inequality \eqref{3}. The following Lemmas will play a role later. 
\begin{lemma}\label{5}
	\cite[Lemma 3.2]{3} Let $A,B\in \mathcal{B}\left( \mathcal{H} \right)$ be two positive operators satisfying $0<m\le A,B\le M$ for some scalars. If $f:\left[ m,M \right]\to \mathbb{R}$ is a continuous concave function, then for each $v\in \left[ 0,1 \right]$
\[\beta \left( m,M,f \right)\le f\left( A \right){{\nabla }_{v}}f\left( B \right)-f\left( A{{\nabla }_{v}}B \right)\le -\beta \left( m,M,f \right)\]
	where
	\begin{equation}\label{6}
\beta \left( m,M,f \right)=\min \left\{ \frac{f\left( M \right)-f\left( m \right)}{M-m}\left( t-m \right)+f\left( m \right)-f\left( t \right):\text{ }m\le t\le M \right\}.
	\end{equation}
\end{lemma}
\begin{lemma}\label{4}
	\cite{1} Let $A\in \mathcal{B}\left( \mathcal{H} \right)$ such that $m\le A\le M$ and $\Phi $ be a normalized positive linear map. If $f:\left[ m,M \right]\to \mathbb{R}$ is a concave function, then
\[-\widetilde{\beta }\left( m,M,f \right)\le f\left( \Phi \left( A \right) \right)-\Phi \left( f\left( A \right) \right)\le \widetilde{\beta }\left( m,M,f \right)\]
where
\begin{equation}\label{7}
\widetilde{\beta }\left( m,M,f \right)=\max \left\{ f\left( t \right)-\frac{f\left( M \right)-f\left( m \right)}{M-m}\left( t-m \right)-f\left( m \right):\text{ }m\le t\le M \right\}.
\end{equation}
\end{lemma}

\begin{theorem}
	Let $A,B\in \mathcal{B}\left( \mathcal{H} \right)$ be two self-adjoint operators such that $m\le A,B\le M$ and $\Phi $ be normalized positive linear map. If $f:\left[ m,M \right]\to \mathbb{R}$ is a concave function, then
\[f\left( \Phi \left( A{{\nabla }_{v}}B \right) \right)\le \Phi \left(f\left ( A\right)\right) {{\nabla}_{v}\Phi \left(f\left(B\right)\right)+2\widetilde{\beta }\left( m,M,f \right)}\]
where   $\widetilde{\beta }\left( m,M,f \right)$  is  defined as   \eqref{7}.
\end{theorem}
\begin{proof}
	We have
\[\begin{aligned}
 f\left( \Phi \left( A \right){{\nabla }_{v}}\Phi \left( B \right) \right)&=f\left( \Phi \left( A{{\nabla }_{v}}B \right) \right) \\ 
& \le \widetilde{\beta }\left( m,M,f \right)+\Phi \left( f\left( A{{\nabla }_{v}}B \right) \right) \quad \text{(by Lemma \ref{4})} \\ 
& \le \widetilde{\beta }\left( m,M,f \right)-\beta \left( m,M,f \right)+\Phi \left( f\left( A \right){{\nabla }_{v}}f\left( B \right) \right)  \quad \text{(by Lemma \ref{5}).}
\end{aligned}\]
The proof is completed. 
\end{proof}
\begin{corollary}
Let $A,B\in \mathcal{B}\left( \mathcal{H} \right)$ be two contractions operators such that $m\le A,B\le M$ and $\Phi $ be normalized positive linear map. Then for any  $v\in[0,1]$ and $r\notin[-1,0]\cup [1,2]$,
\[{{\left( I-\Phi \left( A{{\nabla }_{v}}B \right) \right)}^{r}}\le \Phi \left( {{\left( I-A \right)}^{r}}{{\nabla }_{v}}{{\left( I-B \right)}^{r}} \right)+2\widetilde{\beta }\left( m,M,(1-t)^r \right).\]

\end{corollary}

\vskip 0.9 true cm

{\tiny $^{1,2,3}$Department of Mathematics, Mashhad Branch, Islamic Azad University, Mashhad, Iran.
\vskip 0.2 true cm		
	{\it E-mail address:} shiva.sheybani95@gmail.com 
\vskip 0.1 true cm	
	{\it E-mail address:} erfanian@mshdiau.ac.ir
\vskip 0.1 true cm	
{\it E-mail address:} khanehgir@mshdiau.ac.ir


\end{document}